\documentclass[10pt]{article}
\usepackage{amsmath}                        
\usepackage{amsthm}                         
\usepackage{amssymb}                        
\usepackage{amsfonts}    
\usepackage[english]{babel}                   
\usepackage{latexsym}                       
\usepackage{setspace}						
\usepackage[colorlinks=true,linkcolor=blue,citecolor=blue]{hyperref}
\usepackage[all]{xy}

\theoremstyle{plain}

\newtheorem{lem}{Lemma}[section]
\newtheorem{thm}[lem]{Theorem}
\newtheorem{defin}{Definition}
\newtheorem{prop}[lem]{Proposition}
\newtheorem{coro}[lem]{Corollary}
\newtheorem{remar}[lem]{Remark}
\newtheorem{ejemplo}[lem]{Example}

\newcommand{\R}{\mathbb{R}}
\newcommand{\C}{\mathbb{C}}

\newcommand{\la}{\langle}
\newcommand{\ra}{\rangle}
\title{\bf Timelike surfaces in Minkowski space with a canonical null direction}
\author{Victor H. Patty-Yujra\footnote{email: victorp@im.unam.mx, Instituto de Matem\'aticas UNAM, Unidad Juriquilla. Quer\'etaro, M\'exico.} \  and \  Gabriel Ruiz-Hern\'andez\footnote{email: gruiz@matem.unam.mx, Instituto de Matem\'aticas UNAM, Unidad Juriquilla. Quer\'etaro, M\'exico.}}

\begin{document}
\maketitle

\begin{abstract} Given a constant vector field $Z$ in Minkowski space, a timelike surface is said to have a  canonical null direction with respect to $Z$ if the projection of $Z$ on the tangent space of the surface gives a lightlike vector field. In this paper we describe these surfaces in the ruled case. For example when the Minkowski space has three dimensions then a surface
with a canonical null direction is minimal and flat.
On the other hand, we describe several properties in the non ruled case and we partially describe these surfaces in four-dimensional Minkowski space. We give different ways for building these surfaces in four-dimensional Minkowski space and we finally use the Gauss map for describe another properties of these surfaces.
\end{abstract}



\section*{Introduction}
We consider $\R^{n,1}$ the $(n+1)-$dimensional Minkowski space defined by $\R^{n+1}$ endowed with the metric of signature $(n,1)$
$$\la\cdot,\cdot\ra=-dx_1^2+dx_2^2+\ldots+dx_{n+1}^2.$$
A surface $M$ in $\R^{n,1}$ is said to be timelike if the metric $\la\cdot,\cdot\ra$ induces a Lorentzian metric, \textit{i.e.} a metric of signature $(1,1),$ on $M.$

\begin{defin} 
\label{defin:direccion-nula-canonica} \em
We say that a timelike surface $M$ in $\R^{n,1}$ has a canonical null direction with respect to a constant vector field $Z$ in $\R^{n,1}$ if the tangent part $Z^\top$ of $Z$ is a lightlike vector field along $M,$ \textit{i.e.} $Z^{\top}$ is nonzero and $\la Z^{\top},Z^{\top}\ra=0.$ 
We will say that $Z$ defines a null direction on the surface.
\end{defin}

In this paper, we are interested in the description of timelike surfaces with a canonical null direction in Minkowski space. We will begin by describing the compatibility equations which determine a canonical null direction on a surface and we will see that there exists two different cases for consider: the ruled and the non ruled case. We give a complete description of these surfaces in the ruled case (Theorem \ref{caso 1}). On the other hand, we give several properties in the non ruled case and we partially describe these surfaces in four-dimensional Minkowski space (Proposition \ref{teo 2} and Theorem \ref{suma curvas}). We also give different ways for building these surfaces in four-dimensional Minkowski space and we finally use the Gauss map for describe another properties of these surfaces.

The notion of a canonical null direction only makes sense for timelike submanifolds in the $n+1$-dimensional Minkowski space 
and it is inspired in the concept of surfaces with canonical principal direction with respect to 
a parallel vector field defined by F. Dillen and his collaborators in  \cite{DiFaJo} and \cite{DiMuNi}.
The second author together with E. Garnica and O. Palmas in \cite{GPR} investigated the case of hypersurfaces with a canonical principal direction with respect to 
a closed conformal vector field. 

The paper is organized as follows. In Section \ref{sec ecua comp} we describe the compatibility equations which determine a canonical null direction on a timelike surface and we give some properties about their geometry. In Corollary \ref{coro:mean-parallel-implies-minimal} we proved that if a surface in $\R^{n,1}$ has parallel mean curvature then it is minimal.
In Section \ref{caso reglado} we give a classification of these surfaces in Minkowski space in the ruled case. In Section \ref{caso-no-reglado} we study the non ruled case: we give some properties and we partially describe these surfaces in four-dimensional Minkowski space.

\section{The compatibility equations}\label{sec ecua comp}
We consider a timelike surface $M$ in $\R^{n,1}$ with a canonical null direction $Z.$ We can assume that $Z$ is a unit spacelike vector field; therefore, using the natural decomposition $Z=Z^{\top}+Z^{\perp}$ and 
since $\la Z^\top , Z^\top \ra = 0 $ we have that $\la Z^\perp , Z^\perp \ra = 1.$ Here and below we denote by $\la\cdot,\cdot\ra$ the metric on the Minkowski space, on $TM$ and on the normal bundle $NM.$ 

We will denote by $II:TM\times TM \to NM$ the second fundamental form of the immersion $M\subset \R^{n,1}$ given by 
\begin{equation*}
II(X,Y)=\overline{\nabla}_XY-\nabla_XY,
\end{equation*} where $\overline{\nabla}$ and $\nabla$ are the Levi Civita connections of $\R^{n,1}$ and $M,$ respectively. As usual, if $\nu\in NM,$ $A_\nu:TM \to TM$ stands for the symmetric operator such that \begin{equation*}
\la A_\nu (X),Y\ra=\la II(X,Y),\nu \ra,
\end{equation*} for all $X,Y\in TM.$ Finally, we denote by
$\nabla^{\perp}$ the Levi Civita connection of the normal bundle $NM.$ The following lemma is fundamental.

\begin{lem} We have
\begin{equation}\label{ecua comp}
\nabla_XZ^{\top}=A_{Z^{\perp}}(X) \hspace{0.2in}\mbox{and}\hspace{0.2in} \nabla^{\perp}_XZ^{\perp}=-II(Z^{\top},X),
\end{equation}
for all $X\in TM.$ 
\end{lem}
\begin{proof}
Using the Gauss and Weingarten equations, we obtain that
\begin{align*}
0 = \overline{\nabla}_X Z &=  \overline{\nabla}_X  Z^\top +  \overline{\nabla}_X Z^\perp \\
&= \nabla_X Z^\top - A_{Z^{\perp}}(X) + II(Z^{\top},X) + \nabla^{\perp}_XZ^{\perp};
\end{align*} the result follows by taking tangent and normal parts.
\end{proof}

\begin{lem}\label{operador de forma} We have $$A_{Z^{\perp}}(Z^{\top})=0 \hspace{0.2in}\mbox{and}\hspace{0.2in} \nabla_{Z^\top} Z^\top =0.$$ In particular, $Z^{\top}$ is a canonical principal direction on the surface.
\end{lem}
\begin{proof} 
Using \eqref{ecua comp} we get
$$\la A_{Z^{\perp}}(Z^{\top}),X\ra=\la II(Z^{\top},X),Z^{\perp}\ra
=-\la \nabla^{\perp}_XZ^{\perp},Z^{\perp}\ra=-\frac{1}{2}X\la Z^{\perp},Z^{\perp}\ra=0,$$ 
for all $X\in TM.$ 
Finally, $\nabla_{Z^{\top}} Z^{\top}\ =A_{Z^{\perp}}(Z^{\top}) =0$.
\end{proof}

Let us consider $W$ a lightlike vector field tangent to $M$ (\textit{i.e.} $W$ is nonzero and $\la W,W\ra=0$) such that $\la Z^{\top},W\ra=-1.$ 

\begin{remar}\label{curvatura media y gauss} 
If we consider the frame $(Z^{\top},W)$ of lightlike vector fields on $TM$ (with $\la Z^{\top},W\ra=-1$), the mean curvature vector of the immersion is given by $$\vec{H}:=\frac{1}{2}\mbox{tr}_{\la,\ra} II=-II(Z^{\top},W).$$ 
\end{remar}

We define the function $a:=\la II(W,W),Z^{\perp}\ra.$

\begin{lem}\label{simbolos} The Levi-Civita connection of $M$ satisfies the following relations: 
$$ \nabla_{Z^{\top}}Z^{\top}=0=\nabla_{Z^{\top}}W,\hspace{0.2in} \nabla_WZ^{\top}=-aZ^{\top}\hspace{0.2in}\mbox{and}\hspace{0.2in} \nabla_WW=aW.$$ In particular, $[Z^{\top},W]=aZ^{\top}.$
\end{lem}
\begin{proof} The first equality was given in Lemma \ref{operador de forma}. 
Now, $\la W,W\ra=0$ implies that $\la \nabla_{Z^{\top}}W,W\ra=0;$ and $\la Z^{\top},W\ra=-1$ 
implies $0=\la\nabla_{Z^{\top}}Z^{\top},W\ra+\la Z^{\top},\nabla_{Z^{\top}}W \ra=\la Z^{\top},\nabla_{Z^{\top}}W \ra;$ therefore, $$\nabla_{Z^{\top}}W =-\la\nabla_{Z^{\top}}W,W\ra Z^{\top}-\la\nabla_{Z^{\top}}W,Z^{\top}\ra W=0.$$ 
In a similar way, using $\la Z^{\top},Z^{\top}\ra=0$ we deduce that $\la\nabla_WZ^{\top},Z^{\top}\ra=0;$ 
using \eqref{ecua comp} we get $\la\nabla_WZ^{\top},W\ra=\la A_{Z^{\perp}}(W),W\ra=\la II(W,W),Z^{\perp}\ra=a;$ thus, $$\nabla_WZ^{\top} =-\la\nabla_WZ^{\top},W\ra Z^{\top}-\la\nabla_WZ^{\top},Z^{\top}\ra W=-aZ^{\top}.$$ 
On the other hand, since  $\la \nabla_WW,W\ra=0,$ and 
$\la\nabla_WW,Z^{\top}\ra=-\la W,\nabla_WZ^{\top}\ra=\la W,aZ^{\top}\ra=-a,$ 
we deduce that, 
$$\nabla_WW=-\la\nabla_WW,W\ra Z^{\top}-\la\nabla_WW,Z^{\top}\ra W=aW.$$ 
Finally, $[Z^{\top},W]=\nabla_{Z^{\top}}W-\nabla_WZ^{\top}=aZ^{\top},$ because $\nabla_{Z^{\top}}W=0.$
\end{proof}

We have the following relations for the curvature tensors of $M$.

\begin{prop}\label{tensor} 
The curvature tensor $R$ and the normal curvature tensor $ R^{\perp}$ of $M$ in $\R^{n,1}$ are given by 
$$R(Z^{\top},W)Z^{\top}=Z^{\top}(a)Z^{\top} \hspace{0.2in}\mbox{and}\hspace{0.2in} R^{\perp}(Z^{\top},W)Z^{\perp}=aII(Z^{\top},Z^{\top}).$$
\end{prop}
\begin{proof} 
Using the equalities of Lemma \ref{simbolos}, we get
\begin{align*}
R(Z^{\top},W)Z^{\top} &= \nabla_W\nabla_{Z^{\top}}Z^{\top}-\nabla_{Z^{\top}}\nabla_WZ^{\top} +\nabla_{[Z^{\top},W]}Z^{\top} \\
&= -\nabla_{Z^{\top}}(-aZ^{\top})+\nabla_{(aZ^{\top})}Z^{\top} \\
&= Z^{\top}(a)Z^{\top}.
\end{align*}
On other hand, by \eqref{ecua comp} we have
\begin{align*}
R^{\perp}(Z^{\top},W)Z^{\perp} &= \nabla^{\perp}_W\nabla^{\perp}_{Z^{\top}}Z^{\perp}-\nabla^{\perp}_{Z^{\top}}\nabla^{\perp}_W Z^{\perp} +\nabla^{\perp}_{[Z^{\top},W]}Z^{\perp} \\
&=-\nabla^{\perp}_W(II(Z^{\top},Z^{\top})) + \nabla^{\perp}_{Z^{\top}}(II(W,Z^{\top}))-aII(Z^{\top},Z^{\top});
\end{align*} by Codazzi equation and the equalities of Lemma \ref{simbolos}, we obtain that
\begin{align*}
 &-\nabla^{\perp}_W(II(Z^{\top},Z^{\top})) + \nabla^{\perp}_{Z^{\top}}(II(W,Z^{\top})) \\
&= -\left(\tilde{\nabla}_W II \right)(Z^{\top},Z^{\top})-II(\nabla_WZ^{\top},Z^{\top})-II(Z^{\top},\nabla_WZ^{\top}) \\ 
& \ \ \ \  + \left(\tilde{\nabla}_{Z^{\top}} II \right)(W,Z^{\top}) +II(\nabla_{Z^{\top}}W,Z^{\top})+II(W,\nabla_{Z^{\top}}Z^{\top}) \\
&=2aII(Z^{\top},Z^{\top}),
\end{align*} this finish the proof.
\end{proof}

\begin{coro}\label{curv gauss} 
The Gaussian curvature of 
$M$ is given by $$K=\frac{\la R(Z^{\top},W)Z^{\top},W\ra}{|Z^{\top}|^2|W|^2-\la Z^{\top},W\ra^2}=Z^{\top}(a).$$
\end{coro}

Using the formula above for the Gauss curvature $K,$ we will find a relation between the norm of the mean curvature
vector and the Gaussian curvature. 

\begin{prop}\label{paralelo minima}
The mean curvature vector and its derivative satisfies the following relations:
\begin{equation}\label{formu 0} \nabla_W^{\perp}\vec{H} =-\nabla_{Z^{\top}}^{\perp}(II(W,W)) \hspace{0.2in}\mbox{and}\hspace{0.2in}
|\vec{H}|^2 =-\la\nabla_W^{\perp}\vec{H},Z^{\perp} \ra. 
\end{equation} 
Moreover, we have $$K=|\vec{H}|^2- \la II(W,W),II(Z^{\top},Z^{\top})\ra.$$
\end{prop}
\begin{proof}
By Codazzi equation and the formulae of Lemma \ref{simbolos}, we have
\begin{align*} 
\nabla_{Z^{\top}}^{\perp}(II(W,W))&=\left( \tilde{\nabla}_{Z^{\top}}II\right)(W,W)+2II(\nabla_{Z^{\top}}W,W) \notag\\
&= \left( \tilde{\nabla}_WII\right)(Z^{\top},W) \notag \\
&= \nabla_W^{\perp}(II(Z^{\top},W))-II(\nabla_WZ^{\top},W)-II(Z^{\top},\nabla_WW) \notag\\
&= -\nabla_W^{\perp}\vec{H}+aII(Z^{\top},W)-aII(Z^{\top},W) \notag \\
&= -\nabla_W^{\perp}\vec{H}.
\end{align*}
On other hand, since
$\la\vec{H},Z^{\perp}\ra=-\la II(Z^{\top},W),Z^{\perp}\ra=0$ (see Lemma \ref{operador de forma}), from \eqref{ecua comp} we get
\begin{align*}\label{paralelo minima} 
0= W\la \vec{H},Z^{\perp}\ra =\la\nabla_W^{\perp}\vec{H},Z^{\perp} \ra +\la\vec{H},\nabla_W^{\perp}Z^{\perp}\ra  
&=\la\nabla_W^{\perp}\vec{H},Z^{\perp} \ra-\la\vec{H},II(Z^{\top},W)\ra \\ 
&=\la\nabla_W^{\perp}\vec{H},Z^{\perp} \ra+|\vec{H}|^2. 
\end{align*} 
Therefore, by Corollary \ref{curv gauss} and the equalities in \eqref{ecua comp}-\eqref{formu 0} we obtain
\begin{align*}
K =Z^{\top}(a)&=Z^{\top}\la II(W,W),Z^{\perp}\ra \\
&= \la\nabla^{\perp}_{Z^{\top}}(II(W,W)),Z^{\perp} \ra+\la II(W,W),\nabla_{Z^{\top}}^{\perp}Z^{\perp}\ra  \\
&= -\la\nabla_W^{\perp}\vec{H},Z^{\perp} \ra - \la II(W,W),II(Z^{\top},Z^{\top})\ra 
\end{align*} which proves the assertion.
\end{proof}

\begin{coro} 
\label{coro:mean-parallel-implies-minimal}
If the mean curvature vector $\vec{H}$ is parallel then the surface $M$ is minimal, \textit{i.e.} $\vec{H}=0.$
\end{coro}
\begin{proof}
This is a consequence of the second equality in Proposition \ref{paralelo minima}.
\end{proof}

The normal curvature tensor $R^{\perp}$ is determined by the vector $II(Z^{\top},Z^{\top}),$ which is orthogonal to  $Z^{\perp}$ (see the proof of Lemma \ref{operador de forma}: $\la II(Z^{\top},Z^{\top}),Z^{\perp}\ra=0$). Therefore, we can consider two cases: when $II(Z^{\top},Z^{\top})=0$ (the ruled case) and when $II(Z^{\top},Z^{\top})\neq 0$ (the non ruled case).

\section{The ruled case}\label{caso reglado}
In this section we study the case of a timelike surface $M$ in $\R^{n,1}$ with a canonical null direction $Z$ such that $II(Z^{\top},Z^{\top})=0.$ By Remark \ref{curvatura media y gauss} and Proposition \ref{tensor}, 
the Gauss curvature and the normal curvature tensor satisfy the following relations:
\begin{equation}\label{umbi}
|\vec{H}|^2-K=0 \hspace{0.3in}\mbox{and}\hspace{0.3in} R^{\perp}=0.
\end{equation}  
The timelike surfaces in four-dimensional pseudo Euclidean space for which \eqref{umbi} is valid
are called umbilic (if $II-\la\cdot,\cdot\ra\vec{H}=0$) or quasi-umbilic (if $II-\la\cdot,\cdot\ra\vec{H}\neq 0$).
See e.g. \cite{bayard_sanchez,bayard_patty_sanchez}. \\
The surfaces in $\R^{2,1}$ such that $|\vec{H}|^2-K=0$ were classified in \cite{Cle}.

\begin{remar}\label{normal paralelo} 
The normal vector field $Z^{\perp}$ is parallel if and only if $M$ is minimal. 
Let us verify this fact.
By \eqref{ecua comp}, we have
$$\nabla^{\perp}_{Z^{\top}}Z^{\perp}=-II(Z^{\top},Z^{\top})=0 \hspace{0.2in}\mbox{and}\hspace{0.2in} \nabla^{\perp}_{W}Z^{\perp}=-II(Z^{\top},W)=\vec{H},$$ 
which proves the assertion.
\end{remar}

The next result gives a local description of a timelike surface $M$ in $\R^{n,1}$ 
with a canonical null direction $Z$ such that $II(Z^{\top},Z^{\top})=0.$
We moreover assume that $Z$ is not orthogonal to the surface; otherwise,
$M$ should be any surface in a hyperplane orthogonal to $Z$.

\begin{thm}\label{caso 1} 
A timelike surface $M$ in $\R^{n,1}$ has a canonical null direction
with respect to $Z$ and satisfies the condition $II(Z^{\top},Z^{\top})=0$ 
if and only if $M$ 
can be locally parametrized by
\begin{equation}\label{local1} 
\psi(x,y)=\alpha(x)+y\ Z^{\top}(x),
\end{equation} 
where $\alpha(x)$ is a lightlike curve in $\R^{n,1},$ $Z^{\top}(x)$ is the restriction of the null vector field $Z^{\top}$ along $\alpha$ and where the following conditions holds
\begin{itemize}
\item $Z$ is not orthogonal to $\alpha'(x)$ for every $x,$
\item the vectors $\alpha'(x)$ and $Z^{\top}(x)$ are linearly independent for every $x,$
\item the position vectors $Z^\top(x)$  gives a curve in a timelike hyperplane.
\end{itemize}
\end{thm}
\begin{proof} 
Let us consider a coordinate system
$(x,y)\mapsto \psi(x,y)$ of $M$
such the metric of $M$ is given by
$$\la\cdot,\cdot\ra=-2\lambda(x,y) dx dy,$$ 
where $\lambda$ is some positive function; we moreover assume that 
$Z^{\top}=\frac{\partial \psi}{\partial y}$ satisfies $II(Z^{\top},Z^{\top})=0.$
By calculating the Christoffel symbols of the metric we get that 
$$\nabla_{Z^{\top}}Z^{\top}=\frac{1}{\lambda}\frac{\partial \lambda}{\partial y} Z^{\top}.$$ 
Thus, $T:=\frac{1}{\lambda}Z^{\top}$ satisfies that $\nabla_{Z^{\top}}T=0$ and $II(T,Z^{\top})=0.$  
Since $\overline{\nabla}_{Z^{\top}}T=\nabla_{Z^{\top}} T+II(T,Z^{\top})=0,$ we have that
$$T(\psi(x,y))=T(\psi(x,0))+\int_0^y \frac{\partial}{\partial u}(T(\psi(x,u)))du=T(\psi(x,0)).$$ 
Then, 
$Z^{\top}(\psi(x,y))=\frac{\lambda(x,y)}{\lambda(x,0)}Z^{\top}(\psi(x,0)),$
and therefore, 
\begin{align*} \psi(x,y) &=\psi(x,0)+\int_0^y \frac{\partial \psi}{\partial u}(x,u)du \\ 
&=\psi(x,0)+ \left( \int_0^y \frac{\lambda(x,u)}{\lambda(x,0)}du \right) Z^{\top}(\psi(x,0)).
\end{align*} 
So,  $\psi$ can be written as $$\psi(x,y)=\alpha(x)+f(x,y)\ Z^{\top}(x),$$ where
$\alpha(x):=\psi(x,0)$ is a lightlike curve in $\R^{n,1}$ and $Z^{\top}(x):=Z^{\top}(\psi(x,0))$ 
is a lightlike vector field along $\alpha.$ 
Since $\psi(x,0)=\alpha(x)$, we have that
$f(x,0)=0;$  moreover $$\frac{\partial f}{\partial y}=\frac{\lambda(x,y)}{\lambda(x,0)}>0.$$ 
So, the formulae $x'=x$ and $y'=f(x,y),$ 
define local coordinates such that \eqref{local1} is valid.
Moreover, since $Z=Z^{\top}+Z^{\perp},$ we have that $Z$ is a spacelike constant vector 
with $\la Z, Z^\top  \ra =0,$ in particular 
$\la Z, Z^\top(x)  \ra =0 $ for all $x;$ thus
the positions vectors $ Z^\top (x)$
are orthogonal to $Z$ and so
they are contained in the timelike hyperplane orthogonal to $Z.$

Reciprocally, suppose that $M$ is parametrized as in \eqref{local1}.
Since the positions vectors $ Z^\top (x)$ lives in a timelike hyperplane,
we can choose a constant spacelike vector in the spacelike line orthogonal
to the hyperplane. So, $ \la Z, Z^\top(x)  \ra =0 $ for all $x.$
This implies that the tangent part of $Z$ is $\frac{\partial \psi}{\partial y}=Z^\top (x)$ 
because $Z$ is not orthogonal to $\alpha'(x)$.
Finally, since $M$ is a ruled surface with rules in the direction $Z^\top(x),$
we deduce that $II(Z^{\top},Z^{\top})=0$.
\end{proof}

\subsection{Timelike surfaces in $\R^{2,1}$}

In this case, the normal vector $Z^{\perp}$ is parallel; by \eqref{ecua comp} we have that $II(Z^{\top},X)=0,$ 
for all $X\in TM;$ see Remark \ref{normal paralelo}. Using moreover \eqref{umbi} we get:

\begin{coro}
A timelike surface $M$ in $\R^{2,1}$ 
with a canonical null direction
$Z$ is flat and minimal. 
\end{coro} 

\begin{thm} 
A timelike surface $M$ in $\R^{2,1}$ with a canonical null direction $Z$ can be locally parametrized by
\begin{equation}\label{local2}
\psi(x,y)=\alpha(x)+y\ T_0,
\end{equation} where $\alpha(x)$ 
is a lightlike curve in $\R^{2,1},$ $T_0$ 
is some constant lightlike vector along $\alpha,$ 
and the vectors $\alpha'(x)$ and $T_0$ 
are linearly independent for every $x.$
\end{thm}
\begin{proof} 
Let us observe that in this case we have that
$II(Z^{\top},Z^{\top})=0.$ 
We can adapt the proof of Proposition \ref{caso 1} to obtain
that $M$ can be locally parametrized as in \eqref{local1}. 
The second fundamental form in the coordinates 
$(x,y)$ is given by $II=(d^2\psi)^N.$ 
The mean curvature vector
$\vec{H}=\frac{1}{2}g^{ij}II_{ij}=\frac{1}{\la \alpha',Z^{\top}\ra}\left\lbrace (Z^{\top})'\right\rbrace^N$ 
satisfies the relation $$K=|\vec{H}|^2=\frac{|(Z^{\top})'|^2}{\la 	\alpha',Z^{\top}\ra^2}.$$ 
therefore, the condition
$K=|\vec{H}|^2=0,$ is equivalent to $|(Z^{\top})'|^2=0.$ Since $|Z^{\top}|^2=0$ and $\la(Z^{\top})',Z^{\top}\ra=0,$ 
we have the relation
$(Z^{\top})'(x)=h(x)Z^{\top}(x).$ 
Thus, by integration we get $Z^{\top}(x)=H(x)Z^{\top}(0)$ 
where $H$ is a smooth function such that $H(0)=1.$ 
Using the change of variable $x'=x,$ $y'=yH(x),$ and writing $T_0:=Z^{\top}(0),$ we find that
$M$ is parametrized by $\alpha(x')+y'T_0,$ for small values of
$x'$ and $y'.$ 

Reciprocally, suppose that $M$ is parametrized as in \eqref{local2}. 
Thus, a spacelike constant vector $Z$ in $\R^{2,1}$ 
such that $\la Z,T_0\ra=0,$ 
defines a canonical null direction on $M.$ Moreover,  
$Z^{\top}(x)=H(x)T_0,$ 
for some smooth function $H(x).$ 
\end{proof}

\section{The non ruled case}\label{caso-no-reglado} 

In this section we study the case of a timelike surface $M$ in $\R^{n,1}$ with a canonical null direction $Z$ such that $II(Z^{\top},Z^{\top})\neq 0.$ We note that, as a consequence of Proposition \ref{tensor} and Corollary \ref{curv gauss}, we have the following: 

\begin{coro}\label{normal plano} 
Let us assume that the surface $M$ has normal curvature tensor $R^{\perp}$ identically zero (\textit{i.e.} the function $a$ is identically zero). Then the Gauss curvature $K$ is also constant zero.
\end{coro}

We note that, if we assume that $\nabla a$ is a multiple of $Z^{\top}$ we get that the Gauss curvature $K=Z^{\top}(a)=\la \nabla a, Z^{\top}\ra$ (Corollary \ref{curv gauss}) is zero. We will describe the converse statement. We need some lemmas.

\begin{lem}\label{function-f} There is a local smooth function $f:M \to \R$ such that $\nabla f= Z^{\top}.$ Moreover, $f$ is a harmonic function, \textit{i.e.} $\Delta f=0.$ 
\end{lem}
\begin{proof} We consider the $1-$form $\theta(X)=\la X,Z^{\top}\ra,$ for all $X\in TM.$ Using the equalities of Lemma \ref{simbolos}, we get $\theta$ is a closed $1-$form, \textit{i.e.} $d\theta=0;$ thus, there exists a function $f:M\to \R$ such that $df=\theta,$ that is $\nabla f=Z^{\top}.$

We compute the laplacian of the function $f.$ In the orthonormal frame $\left(\frac{Z^{\top}+W}{\sqrt{2}},\frac{Z^{\top}-W}{\sqrt{2}}\right)$ on $TM,$ we get 
\begin{equation*}
\Delta f=-2\mbox{Hess}f(Z^{\top},W)=-2 \la \nabla_{Z^{\top}}\nabla f,W \ra=-2\la\nabla_{Z^{\top}}Z^{\top},W \ra=0,
\end{equation*} because $\nabla_{Z^{\top}}Z^{\top}=0$ (see Lemma \ref{operador de forma}).
\end{proof}

\begin{lem}\label{function-a} The laplacian of the function $a=\la II(W,W),Z^{\perp}\ra$ is given by 
\begin{equation*}
\Delta a=-2Ka-2 W(K),
\end{equation*} where $K$ is the Gauss curvature of the surface. In particular, if the Gauss curvature is zero, $a$ is a harmonic function.
\end{lem}
\begin{proof} In the same frame, as in the proof of Lemma \ref{function-f}, we get 
\begin{equation*}
\Delta a=-2 \mbox{Hess}\ f(W,Z^{\top})=-2\la \nabla_W\nabla a,Z^{\top}\ra.
\end{equation*} On the other hand, since $K=Z^{\top}(a)=\la\nabla a,Z^{\top}\ra$ (Corollary \ref{curv gauss}), using Lemma \ref{simbolos} we obtain 
\begin{align*}
W(K)=W\la \nabla a,Z^{\top}\ra &= \la \nabla_W\nabla a, Z^{\top}\ra+\la \nabla a,\nabla_WZ^{\top}\ra 
= -\frac{1}{2}\Delta a-a\la \nabla a,Z^{\top}\ra 
\end{align*} which is the equality of the lemma.
\end{proof}

\begin{prop} The Gauss curvature $K$ is zero if and only if there exists a harmonic function $a_1:M\to \R$ such that $\nabla a=a_1 Z^{\top}.$
\end{prop}
\begin{proof} We assume that the Gauss curvature $K$ is zero: since $K=Z^{\top}(a)=\la\nabla a,Z^{\top}\ra$ (Corollary \ref{curv gauss}), there exists a smooth function $a_1:M\to \R$ such that 
$\nabla a= a_1 Z^{\top}$ because $Z^{\top}$ is a null vector field. Using Lemmas \ref{function-f} and \ref{function-a}, we obtain 
\begin{equation*}
0=\Delta a=\mbox{div}(\nabla a)=\mbox{div}(a_1 \nabla f)=\la \nabla a_1,\nabla f\ra+a_1\Delta f =\la \nabla a_1,Z^{\top}\ra,
\end{equation*} thus, there exists a smooth function $a_2:M\to \R$ such that 
$\nabla a_1=a_2 Z^{\top}.$ The laplacian of the function $a_1$ is given by 
\begin{equation*}
\Delta a_1=-2 \la \nabla_W \nabla a_1,Z^{\top}\ra=-2 \la W(a_2)Z^{\top}+a_2\nabla_WZ^{\top},Z^{\top}\ra=0.
\end{equation*} Note that, we can continue with this procedure.
\end{proof}

\subsection{Timelike surfaces in $\R^{3,1}$} 

In this case, we consider the normalized vector field 
$$\nu:=\frac{II(Z^{\top},Z^{\top})}{|II(Z^{\top},Z^{\top})|}\ \in \ NM.$$ Note that $\nu$ is orthogonal to $Z^{\perp}$ (see Lemma \ref{operador de forma}). We recall that $Z^{\perp}$ is a spacelike vector field with $\la Z^{\perp},Z^{\perp}\ra=1.$ 
So, $(Z^{\perp},\nu)$ defines an oriented orthonormal frame of the normal bundle $NM$ along $M$.

\begin{coro}\label{curv normal} 
The normal curvature of the surface $M$ in $\R^{3,1}$ is given by
$$K_N=a|II(Z^{\top},Z^{\top})|.$$ 
\end{coro}
\begin{proof} 
Using the Ricci equation, in the orthonormal frame
$\left(\frac{Z^{\top}+W}{\sqrt{2}},\frac{Z^{\top}-W}{\sqrt{2}}\right)$ 
on $TM,$ we obtain
\begin{align*} 
K_N &= \left\la ( A_{Z^{\perp}}\circ A_{\nu}-A_{\nu}\circ A_{Z^{\perp}}) \left(\frac{Z^{\top}+W}{\sqrt{2}}\right), \frac{Z^{\top}-W}{\sqrt{2}}\right\ra \\
&= -\la (A_{Z^{\perp}}\circ A_{\nu}-A_{\nu}\circ A_{Z^{\perp}})(Z^{\top}),W \ra \\
&=\la R^{\perp}(Z^{\top},W)Z^{\perp},\nu \ra; 
\end{align*} we get the result by replacing the second equality given in Proposition \ref{tensor}.
\end{proof}

Now, we will give a relation between the Gauss curvature, the normal curvature and the mean
curvature vector of $M$ in $\R^{3,1}.$ 

\begin{lem}\label{relacion 1} 
In the orthonormal frame $(Z^{\perp},\nu)$ orthogonal to $M,$ 
we have the following relation
$$II(W,W)=\frac{K_N}{|II(Z^{\top},Z^{\top})|}Z^{\perp} 
+ \frac{|\vec{H}|^2-K}{|II(Z^{\top},Z^{\top})|} \nu.$$ In particular, 
$|II(W,W)|^2|II(Z^{\top},Z^{\top})|^2=(|\vec{H}|^2-K)^2+K_N^2.$
\end{lem}
\begin{proof} 
We have
\begin{align*} II(W,W)&=\la II(W,W),Z^{\perp}\ra Z^{\perp}+\la II(W,W),\nu\ra\nu \\
&= aZ^{\perp}+ \frac{\la II(W,W),II(Z^{\top},Z^{\top})\ra}{|II(Z^{\top},Z^{\top})|}\nu, 
\end{align*} we get the result by using Corollary \ref{curv normal} and Proposition \ref{paralelo minima}. 
\end{proof}

Using the lemma above we have the following description in a simple case:
\begin{prop}\label{teo 2} 
Consider a timelike surface $M$ in $\R^{3,1}$ with a canonical null direction $Z$ 
such that $II(Z^{\top},Z^{\top})\neq 0.$ If $M$ is minimal and has flat normal bundle (\textit{i.e.} $K_N=0$)
then it can be parametrized as
$$\psi(x,y)=\alpha(x)+y\ W_0,$$ 
where $\alpha'(x)=Z^{\top}(x),$ $\alpha''(x)=II(Z^{\top},Z^{\top})$ ($\alpha$ is a geodesic of $M$), $W_0$ is some constant lightlike tangent vector along $\alpha$ and the vectors $\alpha'(x)$ and $W_0$ are linearly independent for every $x.$
\end{prop}
\begin{proof}
Since $a=0$ (\textit{i.e.} $K_N=0$), by Lemma \ref{simbolos} 
we have that 
$Z^{\top}$ and $W$ 
are parallel vector fields and $[Z^{\top},W]=0.$ 
So, there exists a coordinate system 
$(x,y) \mapsto \psi(x,y)$ of $M$ such that
\begin{equation*} 
\frac{\partial \psi}{ \partial x}(x,y) =Z^{\top}(\psi(x,y)) \hspace{0.2in}\mbox{and}\hspace{0.2in} \frac{\partial \psi}{\partial y}(x,y) =W(\psi(x,y)). 
\end{equation*}
We have that $II(W,\cdot)=0:$ indeed, $II(W,Z^{\top})=-\vec{H}=0$ and $II(W,W)=0$ because $K=K_N=|\vec{H}|^2=0$ in Lemma \ref{relacion 1}. Since $\nabla W=0,$ we get that $\overline{\nabla}W=0;$ thus, 
\begin{equation}\label{depen s}
W(\psi(x,y))=W(\psi(x,0))+\int_0^y \frac{\partial}{\partial u} W(\psi(x,u)) du=W(\psi(x,0)),
\end{equation} 
this implies that, 
\begin{align*} \psi(x,y) = \psi(x,0)+\int_0^y \frac{\partial\psi}{\partial u}(x,u) du  
&= \psi(x,0)+\int_0^y W(\psi(x,0))du \\ 
&= \psi(x,0)+y\ W(\psi(x,0)).
\end{align*}
In the same way, let us observe that
\begin{equation}\label{depen t}
W(\psi(x,y))=W(\psi(0,y))+\int_0^x \frac{\partial}{\partial r} W(\psi(r,y)) dr=W(\psi(0,y)).
\end{equation} 
The equalities
\eqref{depen s}-\eqref{depen t} imply that $W(\psi(x,y))=:W_0$ is constant, 
\textit{i.e.} $$\psi(x,y)=\alpha(x)+y\ W_0,$$ 
where
$\alpha(x):=\psi(x,0)$ 
is a lightlike curve such that
$\alpha'(x)=Z^{\top}(\psi(x,0)).$
\end{proof}

The following example describe a timelike surface in $\R^{3,1}$ with a canonical null direction $Z$ such that $II(Z^{\top},Z^{\top})\neq 0$ which is minimal but has normal curvature not zero. Here and below, we denote by $\{ e_1,e_2,e_3,e_4 \}$ the canonical basis of the four-dimensional Minkowski space; of course $e_1$ is a timelike vector. 

\begin{ejemplo} Let us consider the surface $M$ in $\R^{3,1}$ parametrized as
\begin{equation*}\label{ejemplo-1}
\psi(x,y) = \alpha(x) + \beta(y),
\end{equation*} where $\alpha$ and $\beta$ are two lightlike curves contained in the timelike hyperplanes orthogonal to $e_4$ and $e_3,$ respectively, and satisfy the following conditions
\begin{itemize}
\item $ \la  \alpha'(x),\beta'(y) \ra \neq 0$ for every $(x,y),$
\item $\la e_3,\alpha'(x)\ra\neq 0$ for every $x,$
\item $\beta''(y)$ (resp. $\alpha''(x)$) is not lightlike: in other case, $\beta''(y)$ would be linearly dependent to $\beta'(y)$ and thus $\beta$ would be a lightlike line in $\R^{3,1}.$
\end{itemize}
We have that $M$ is a minimal timelike surface in $\R^{3,1}$ with normal curvature not zero and has a canonical null direction with respect to $e_3$ with $II(e_3^\top,e_3^\top)\neq 0.$ 

Indeed, note that $M$ is a timelike surface because its tangent plane is generated by the linearly independent lightlike tangent vectors $\psi_x=\alpha'(x)$ and $\psi_y=\beta'(y).$ On the other hand, since the curve $\beta$ is orthogonal to $e_3,$ the tangent part of $e_3$ is given by 
\begin{equation*}
e_3^\top=\frac{\la e_3 , \beta'(y) \ra}{\la \alpha'(x) , \beta'(y) \ra } \alpha'(x)   + 
\frac{\la e_3 , \alpha'(x) \ra}{\la \alpha'(x) , \beta'(y) \ra } \beta'(y) = \frac{\la e_3 , \alpha'(x) \ra}{\la \alpha'(x) , \beta'(y) \ra } \beta'(y);
\end{equation*} this proves that $e_3^\top = \lambda(x,y) \beta'(y),$ where 
$\lambda(x,y):=\frac{\la e_3,\alpha'(x)\ra}{\la\alpha'(x),\beta'(y)\ra}$ is not zero,  
is a lightlike direction on the surface $M.$ Since $\nabla_{e_3^\top } e_3^\top =0$ (Lemma \ref{operador de forma}), we obtain 
\begin{equation*}
\nabla_{\beta'(y)}\beta'(y) = -\frac{1}{\lambda(x,y)} \frac{\partial \lambda}{\partial y} \beta'(y).
\end{equation*} Therefore, 
\begin{equation*}
\beta''(y)=\nabla_{\beta'(y)} \beta'(y) + II(\beta'(y) , \beta'(y))=-\frac{1}{\lambda(x,y)} \frac{\partial \lambda}{\partial y} \beta'(y) +  II(\beta'(y) , \beta'(y));
\end{equation*} since $\beta'(y)$ and $\beta''(y)$ are linearly independent, we get that $II(\beta'(y) ,\beta'(y)) \neq 0,$ that is 
\begin{equation*}
II(e_3^\top ,e_3^\top ) = \lambda(x,y)^2 II( \beta'(y),\beta'(y)) \neq 0. 
\end{equation*} Now, since $\psi_{xy}=0$ we get that $\nabla_{\psi_x}\psi_y=0$ and $II(\psi_x,\psi_y)=0,$ in particular, $M$ is minimal. We finally prove that $M$ has normal curvature not zero: we consider the lightlike tangent vector
\begin{equation*}
W:=-\frac{1}{\lambda(x,y)\la\alpha'(x),\beta'(y)\ra} \alpha'(x)=-\frac{1}{\la e_3,\alpha'(x)\ra}\alpha'(x)
\end{equation*} which is such that $\la e_3^\top,W\ra=-1;$ since $\nabla_W e_3^{\top}=-a\ e_3^\top$ (Lemma \ref{simbolos}) we obtain 
\begin{equation*}
a=\frac{1}{\lambda(x,y)\la e_3,\alpha'(x)\ra}\frac{\partial \lambda}{\partial x};
\end{equation*} according to Corollary \ref{curv normal}, $K_N=0$ if and only if $a=0,$ that is, if and only if $\frac{\partial \lambda}{\partial x}=0,$ the last equality is equivalent to 
\begin{equation*}
\lambda=\frac{\la e_3,\alpha''(x)\ra}{\la \alpha''(x),\beta'(y)\ra}
\end{equation*} which is valid when $\alpha''(x)$ is linearly dependent to $\alpha'(x).$ 
We finally give an explicit numerical example of this situation: consider
$$\alpha(x)= (\cosh x, \sinh x,x,0) \hspace{0.2in}\mbox{and}\hspace{0.2in} \beta(y)=(\cosh y, y,0,  \sinh y),$$
defined on a domain for $(x,y)$ where $\la \alpha'(x), \beta'(y) \ra \neq 0.$
\end{ejemplo} 

\subsection{Timelike surfaces in $\R^{3,1}$ as a graph of a function}

In this section, we will study the situation when a surface is given as the graph of a smooth function.

Let $f, g : U \subset \R^2  \to \R$ be two smooth functions and consider the surface 
\begin{equation}\label{superficie-grafica}
M:= \left\lbrace (f(x,y), g(x,y), x, y ) \in \R^{3,1} \ | \ x,y \in U  \right\rbrace \ \subset \ \R^{3,1} 
\end{equation} given as a graph of the function $(x,y)\to (f(x,y),g(x,y)).$ A global parametrization of this surface is given by  
\begin{equation*}
\psi:U\subset \R^2\to\R^{3,1}, \hspace{0.2in} \psi(x,y)=(f(x,y),g(x,y),x,y).
\end{equation*}
The tangent vectors to the surface are $\psi_x=(f_x,g_x,1,0)$ and $\psi_y=(f_y, g_y,0,1),$ and the components of the induced metric $\la \cdot,\cdot\ra$ in $M$ are given by
\begin{equation*}
E:=\la \psi_x , \psi_x \ra =  1 - f_x^2 +g_x^2, \hspace{0.3in} F :=\la \psi_x , \psi_y \ra= -f_x f_y + g_x g_y
\end{equation*} and 
\begin{equation*}
G := \la \psi_y , \psi_y \ra = 1 - f_y^2 +g_y^2.
\end{equation*}
The determinant of this metric is 
\begin{equation*}
\det \la\cdot,\cdot\ra=EG-F^2 =1-|\nabla f|^2 + |\nabla g|^2 - \langle \nabla f , \nabla g \rangle^2 ,
\end{equation*}
where the right hand side is calculated on $\R^2$ with its standard Riemannian flat metric;
in particular, $M$ is a timelike surface if and only if $\det\la\cdot,\cdot\ra < 0.$

\begin{prop}
\label{prop:caracterizacion-grafica-con dnc}
Let $M$ be a timelike surface in $\R^{3,1}$ given as in \eqref{superficie-grafica}. 
Then $M$ has a canonical null direction with respect to $e_4$ (resp. $e_3$) if and only if $\psi_x$ (resp. $\psi_y$) is a lightlike vector field along $M.$ In that situation we have 
\begin{equation*}
e_4^\top = \frac{1}{F} \psi_x \hspace{0.3in} \left( \mbox{resp.} \ \  e_3^\top = \frac{1}{F} \psi_y \right). 
\end{equation*} 
\end{prop}
\begin{proof}
We have to calculate the tangent part of $e_4$ along $M$ (the case for the vector $e_3$ is similar therefore it will be omitted): writing $$e_4^\top = a \psi_x + b \psi_y,$$ we get 
\begin{equation*}
\la e_4^\top  , \psi_x \ra=aE+bF \hspace{0.2in}\mbox{and}\hspace{0.2in} 
\la e_4^\top  , \psi_y \ra= aF+bG;
\end{equation*} therefore
%
$$ \left( \begin{array}{c} 0  \\ 1 \end{array} \right)
= \left( \begin{array}{c}  \langle e_4,\psi_x\rangle   \\  \langle e_4,\psi_y \rangle \end{array} \right)
= \left( \begin{array}{c}  \langle e_4^\top  , \psi_x \rangle \\ \langle e_4^\top  , \psi_y \rangle \end{array} \right)
= \left( \begin{array}{cc} E & F    \\ F & G  \end{array} \right) \left( \begin{array}{c} a \\ b    
\end{array} \right) $$
and thus 
\begin{equation*}
a=\frac{-F}{EG-F^2} \hspace{0.2in}\mbox{and}\hspace{0.2in} b=\frac{E}{EG-F^2}.
\end{equation*}
Finally, we get
\begin{equation*}
\la e_4^\top , e_4^\top \ra= a^2 E + b^2 G +2abF=\frac{E}{EG-F^2}
\end{equation*}
that is, $e_4^\top$ is a lightlike vector field along $M$ if and only if $$E=\la \psi_x , \psi_x \ra = 1 - f_x^2 +g_x^2=0,$$ \textit{i.e.} if and only if $\psi_x$ is a lightlike vector field.
\end{proof}


\begin{thm}\label{suma curvas}
Let $M$ be a timelike surface in $\R^{3,1}$ given as in \eqref{superficie-grafica}. Assume that $M$ has a canonical null direction with respect to $e_3$ and $e_4.$ Then $M$ is minimal if and only if $M$ can be locally parametrized by 
\begin{equation}\label{suma-curvas}
\psi(x,y)=\alpha(x)+\beta(y)
\end{equation} where $\alpha$ and $\beta$ are two lightlike curves contained in the timelike hyperplanes orthogonal to $e_4$ and $e_3,$ respectively.
\end{thm}

\begin{proof}
A global basis for the normal bundle $NM$ is given by the vector fields
$$\xi_1 := (1,0,f_x,f_y) \hspace{0.2in}\mbox{and}\hspace{0.2in} \xi_2 := (0, -1, g_x, g_y).$$
The components of the induced metric in $NM$ are given by 
\begin{equation*}
L := \langle \xi_1, \xi_1 \rangle =  | \nabla f|^2 - 1, \hspace{0.1in} M :=\langle \xi_1, \xi_2 \rangle =\langle \nabla f , \nabla g  \rangle, \hspace{0.1in}
N := \langle \xi_2, \xi_2 \rangle = |\nabla g|^2 + 1,
\end{equation*} and satisfies $LN-M^2> 0.$
%
We are going to calculate the condition for $M$ to be minimal. Using Proposition \ref{prop:caracterizacion-grafica-con dnc}, we have that the tangent vectors $\psi_x$ and $\psi_y$ of $M$ are lightlike; therefore, 
$M$ is minimal if and only if $II(\psi_x, \psi_y) =0.$ In general, for $i, j \in \{x, \ y \},$ we have
$$II(\psi_i, \psi_j) = (\overline{\nabla}_{\psi_i} \psi_j )^\perp  = a\ \xi_1  + b\ \xi_2, $$
where $\overline{\nabla}$ is the Levi Civita connection of $\R^{3,1}$ and 
$$
\left( \begin{array}{c}  a   \\  b  \end{array} \right) =
\frac{1}{LN-M^2} \left( \begin{array}{cc} N & -M    \\ -M & L    \end{array} \right)
\left( \begin{array}{c}  \langle \overline{\nabla}_{\psi_i} \psi_j , \xi_1  \rangle   \\
  \langle \overline{\nabla}_{\psi_i} \psi_j , \xi_2  \rangle   \end{array} \right). $$
Since $\overline{\nabla}_{\psi_i} \psi_j = (f_{ij}, g_{ij}, 0, 0), $ we get 
$$\langle \overline{\nabla}_{\psi_i} \psi_j , \xi_1  \rangle = - f_{ij} \hspace{0.2in}\mbox{and}\hspace{0.2in} \langle \overline{\nabla}_{\psi_i} \psi_j , \xi_2  \rangle = - g_{ij} .$$
We deduce that, 
\begin{align*}
II(\psi_x,\psi_x) &= \frac{-f_{xx}N + g_{xx}M}{LN-M^2} \xi_1 + \frac{f_{xx}M - g_{xx}L}{LN-M^2} \xi_2, \\
II(\psi_y,\psi_y) &= \frac{-f_{yy}N + g_{yy}M}{LN-M^2} \xi_1 + \frac{f_{yy}M - g_{yy}L}{LN-M^2} \xi_2,  \\
II(\psi_x,\psi_y) &= \frac{-f_{xy}N + g_{xy}M}{LN-M^2} \xi_1 + \frac{f_{xy}M - g_{xy}L}{LN-M^2} \xi_2. 
\end{align*}
Therefore, $M$ is minimal if and only if
$$  \left( \begin{array}{cc} -N & M    \\ M & -L  \end{array} \right)
\left( \begin{array}{c}  f_{xy}   \\  g_{xy}   \end{array} \right) =
\left( \begin{array}{c}  -f_{xy}N + g_{xy}M  \\   f_{xy}M - g_{xy}L \end{array} \right)=
\left( \begin{array}{c}  0 \\  0 \end{array} \right);$$
since $LN-M^2 > 0,$ we obtain that $f_{xy} =0= g_{xy}$. Thus, by integration we get
\begin{equation*}
f(x,y)= \alpha_1(x) + \beta_1(y) \hspace{0.2in}\mbox{and}\hspace{0.2in}  g(x,y)= \alpha_2(x) + \beta_2(y)
\end{equation*}
This implies that $\psi$ can be written as in \eqref{suma-curvas} with 
$\alpha(x)=( \alpha_1(x) , \alpha_2(x)  , x,0)$ (orthogonal to $e_4$) and $\beta(y)=( \beta_1(y), \beta_2(y) , 0,y)$ (orthogonal to $e_3$).
Let us observe that in this case $\psi_x$ and $\psi_y$ are lightlike vectors if and only if
$\alpha$ and $\beta$ are lightlike curves.
\end{proof}

We consider the isometric embedding of $\R^{2,1}$ in $\R^{3,1}$ given by 
\begin{equation*}
\R^{2,1}:=(e_4)^{\perp},
\end{equation*} where $e_4$ is the fourth vector of the canonical basis of $\R^{3,1}.$

\begin{prop}\label{prop:graficas-con-dnc}
Let $M_0$ be a Lorentzian surface in $\R^{2,1},$ $f: M_0 \to \R$ be a given smooth function. Let us consider the surface obtained as the graph of $f,$ \textit{i.e.}
$$M:= \left\lbrace (p,f(p)) | \ p \in M_0 \right\rbrace\  \subset \ \R^{3,1},$$
with the induced metric. Then $M$ has a canonical null direction with respect to
$e_4$ if and only if $\nabla f$ is a lightlike vector field on $M_0.$
\end{prop}
\begin{proof}
The surface $M$ is parametrized by the immersion
\begin{equation*}
\psi:M_0 \to \R^{3,1}, \hspace{0.3in} \psi(p)=(p,f(p)).
\end{equation*}
We consider a local orthonormal frame $(X_1, X_2)$ of $TM_0$ with $\epsilon_j= \langle X_j, X_j \rangle$ and such that $\epsilon_1 \epsilon_2 = -1$ ($M_0$ is Lorentzian). Moreover, 
%
$X_1$ and $X_2$ are orthogonal to $e_4.$ Using the immersion $\psi,$ we can get the induced
local frame on $TM,$
$$ Y_j:= d\psi(X_j) = X_j + df(X_j) e_4,\hspace{0.3in} j=1, \ 2.$$ So, the induced metric $\la\cdot,\cdot\ra$ on $M$ is given by the
matrix
$$ \la\cdot,\cdot\ra= \left( \begin{array}{cc}
\langle Y_1 , Y_1  \rangle & \langle Y_1 , Y_2  \rangle    \\
\langle Y_2 , Y_1  \rangle & \langle Y_2 , Y_2  \rangle       
\end{array} \right) =
\left( \begin{array}{cc} \epsilon_1 + \langle \nabla f , X_1  \rangle^2 & \langle \nabla f , X_1  \rangle \langle \nabla f , X_2  \rangle    \\ \langle \nabla f , X_1 \rangle  \langle \nabla f , X_2  \rangle &  \epsilon_2 +  \langle  \nabla f , X_2  \rangle^2        \end{array} \right),$$
and its determinant is
\begin{align*}
\det\la\cdot,\cdot\ra = -1+ \epsilon_2 \la \nabla f , X_1  \ra^2  + \epsilon_1  \la \nabla f,X_2 \ra^2 
&=- (1+ \epsilon_1 \la \nabla f , X_1  \ra^2  + \epsilon_2  \la \nabla f , X_2  \ra^2) \\
&= - (1+ \langle \nabla f, \nabla f  \rangle)
\end{align*}
%
(since $\nabla f = \epsilon_1  \langle \nabla f , X_1  \rangle X_1 +   \epsilon_2 \langle \nabla f , X_2  \rangle X_2 $); therefore, $M$ is a timelike surface if and only if $\la \nabla f, \nabla f  \ra > -1.$ 
On the other hand, we have $e_4^\top = a Y_1 + b Y_2$ where
$$
\left(
\begin{array}{c}
 a   \\
 b        
\end{array}
\right)
=
\frac{1}{\det\la\cdot,\cdot\ra}
\left(
\begin{array}{cc}
 \epsilon_2 +  \langle  \nabla f , X_2  \rangle^2     & - \langle \nabla f , X_1  \rangle \langle \nabla f , X_2  \rangle    \\
- \langle \nabla f , X_1 \rangle  \langle \nabla f , X_2  \rangle &   \epsilon_1 + \langle \nabla f , X_1  \rangle^2
\end{array}
\right)
\left(
\begin{array}{c}
 \langle e_4, Y_1   \rangle   \\
 \langle e_4 , Y_2  \rangle
\end{array}
\right);
$$
since, $\langle e_4, Y_1   \rangle =\langle \nabla f , X_1  \rangle $ and $\langle e_4, Y_2   \rangle =\langle \nabla f , X_2  \rangle $
we obtain that
$$
\left(
\begin{array}{c}
 a   \\
 b        
\end{array}
\right)
=
\frac{1}{\det\la\cdot,\cdot\ra}
\left(
\begin{array}{c}
 \epsilon_2   \langle \nabla f , X_1  \rangle  \\
     \epsilon_1   \langle \nabla f , X_2  \rangle    
\end{array}
\right)
=
\frac{-1}{\det\la\cdot,\cdot\ra}
\left(
\begin{array}{c}
 \epsilon_1   \langle \nabla f , X_1  \rangle  \\
  \epsilon_2   \langle \nabla f , X_2  \rangle    
\end{array}
\right),
$$
therefore,
$$e_4^\top = \frac{-1}{\det\la\cdot,\cdot\ra} ( \epsilon_1   \langle \nabla f , X_1  \rangle Y_1 + \epsilon_2   \langle \nabla f , X_2  \rangle Y_2  )    .$$
Thus
\begin{align*}
\la e_4^\top , e_4^\top \ra &= \frac{1}{\det\la\cdot,\cdot\ra^2} ( \langle \nabla f , X_1  \rangle^2 (\epsilon_1 + \langle \nabla f , X_1  \rangle^2) + 
 \langle \nabla f , X_2  \rangle^2 (\epsilon_2 + \langle \nabla f , X_2  \rangle^2) ) \\ 
 & \ \ \ - 2 \langle \nabla f , X_1 \rangle^2  \langle \nabla f , X_2  \rangle^2 \\
 &= \frac{1}{\det\la\cdot,\cdot\ra^2}  ( \la \nabla f, \nabla f  \ra +  \la \nabla f, \nabla f  \ra^2 ).
\end{align*}
Now, it is clear that $e_4^\top$ is a lightlike vector field on $M$ if and only if $\langle \nabla f, \nabla f  \rangle $ is either $0$ or $-1;$ but the case $\langle \nabla f, \nabla f  \rangle = -1$ is not possible because we would have that $\det\la\cdot,\cdot\ra= 0.$
\end{proof}

The following proposition generalize Lemma \ref{function-f}.

\begin{prop}\label{curva-integral-luz}
Let $M$ be a Lorentzian surface, $f: M \to  \R$ be a given smooth function.
If the gradient $\nabla f$ is a lightlike vector field then the integral curves of $\nabla f$ are geodesics and $f$ is a harmonic function, \textit{i.e.} $\triangle f =0 $. 
\end{prop} 
\begin{proof}
Note that $\nabla f \neq 0 $ because it is lightlike vector field, in particular $f$ is not a constant function. By a direct computation we get   
\begin{equation*}
0= X\la\nabla f,\nabla f\ra=2 \la \nabla_X \nabla f, \nabla f \ra = 2 \mbox{Hess}\ f(\nabla f , X), 
\end{equation*}for all $X\in TM;$ 
%
in particular, $\nabla_{\nabla f} \nabla f =0$ because 
$ \langle \nabla_X \nabla f, \nabla f   \rangle =  \langle \nabla_{\nabla f} \nabla f, X  \rangle$. On the other hand, let $W$ be another lightlike vector field defined locally on $M$ such that $\la \nabla f , W  \ra= -1.$ We compute the laplacian of the function $f:$ in the orthonormal frame $ \left(\frac{\nabla f + W}{\sqrt{2}}, \frac{\nabla f - W}{\sqrt{2}} \right) $ on $TM,$ we get 
\begin{equation*}
\Delta f=-2\mbox{Hess}\ f (\nabla f,W)=-2 \la \nabla_{\nabla f}\nabla f,W\ra=0,
\end{equation*} because $\nabla_{\nabla f} \nabla f =0.$
\end{proof}

\begin{ejemplo}
Let us consider the timelike surface $$M_0:= \{(x, \exp iy ) \mid\ x\in \R, \   y \in (0,2 \pi ) \} \ \subset \ \R^{2,1},$$ and 
the function $f : M_0 \to \R$ given by $f(x, \exp iy) = y-x.$ The level curve 
$\gamma(y)= (y-c, \exp iy ) $ is a lightlike geodesic in $M_0$ for all constant $c \in \R.$ Indeed, 
we only have to remark that $\gamma'(y)= (1, i \exp iy ) = \partial_x + \partial_y$ is a lightlike vector field. On the other hand, we compute the gradient of the function $f:$ since $M_0$ is a Lorentzian product, we have that $\partial_x = e_1$ and $\partial_y = -\sin y\ e_2 + \cos y\ e_3$ are an orthonormal frame along $M_0,$ therefore,
$$\nabla f = - (\partial_x f) \partial_x + (\partial_y f) \partial_y = \partial_x + \partial_y = \gamma'(y) , $$
which is a lightlike vector field on $M_0;$ from Proposition \ref{curva-integral-luz} we obtain that $\gamma$ is a geodesic. Finally, by Proposition \ref{prop:graficas-con-dnc}, the timelike surface 
$$M:= \{ (x, \exp iy, y-x) \mid \ x\in\R, \ y\in (0, 2\pi)  \}\ \subset \  \R^{3,1} $$
has a canonical null direction with respect to $e_4.$
\end{ejemplo}

\subsection{Another properties using the Gauss map}
We consider $\Lambda^2\R^{3,1},$ the vector space of bivectors of $\R^{3,1}$ endowed with its natural metric $\la\cdot,\cdot\ra$ of signature $(3,3).$ 
The Grassmannian of the oriented timelike $2-$planes in $\R^{3,1}$ identifies with the submanifold of unit and simple bivectors
\begin{equation*}\label{grassmanniana}
\mathcal{Q}:=\left\lbrace \eta \in\Lambda^2\R^{3,1} \mid \la\eta,\eta\ra=-1,\ \eta\wedge\eta=0   \right\rbrace,
\end{equation*} and the oriented Gauss map of a timelike surface in $\R^{3,1}$ with the map 
\begin{equation*}
G:M \longrightarrow \mathcal{Q}, \hspace{0.2in} p \longmapsto G(p)=u_1\wedge u_2,
\end{equation*} where $(u_1,u_2)$ is a positively oriented orthonormal basis of $T_pM.$ The Hodge $*$ operator $\Lambda^2\R^{3,1}\to\Lambda^2\R^{3,1}$ is defined by the relation 
\begin{equation*}\label{operador estrella}
\la\star \eta,\eta'\ra=\eta\wedge \eta' 
\end{equation*} for all $\eta,\eta'\in \Lambda^2\R^{3,1},$ where we identify $\Lambda^4\R^{3,1}$ to $\R$ using the canonical volume element $e_1\wedge e_2\wedge e_3\wedge e_4$ of $\R^{3,1}.$ It satisfies $*^2=-id_{\Lambda^2\R^{3,1}}$ and thus $i:=-*$ defines a complex structure on $\Lambda^2\R^{3,1}.$ We also define  the map $H: \Lambda^2\R^{3,1} \times \Lambda^2\R^{3,1} \to \C$ by  
\begin{equation}\label{norma H}
H(\eta,\eta')= \la \eta,\eta'\ra+i\ \eta\wedge \eta', 
\end{equation} for all $\eta,\eta'\in \Lambda^2\R^{3,1}.$ This is a $\C-$bilinear map on $\Lambda^2\R^{3,1},$ and we have 
\begin{equation*}
\mathcal{Q}=\left\lbrace \eta\in\Lambda^2\R^{3,1} \mid H(\eta,\eta)=-1 \right\rbrace.
\end{equation*} The bivectors
\begin{equation}\label{base bivectores}
\left\lbrace e_1\wedge e_2,\hspace{0.1in} e_2\wedge e_3,\hspace{0.1in} e_3\wedge e_1 \right\rbrace
\end{equation}  
form an orthomormal basis (with respect to the norm $H$) of $\Lambda^2\R^{3,1}$ as a complex space of signature $(-,+,-).$ Using this basis of $\Lambda^2\R^{3,1}$, the Grassmannian $\mathcal{Q}$ is identifies with a complex hyperboloid of one sheet  \begin{equation*}
\mathcal{Q}\simeq \left\lbrace (z_1,z_2,z_3)\in\C^3 \mid -z_1^2+z_2^2-z_3^2=-1 \right\rbrace.
\end{equation*}

\subsubsection{Timelike surfaces with a canonical null direction}
We consider an oriented timelike surface $M$ in $\R^{3,1}$ with a canonical null direction $Z$ (with $\la Z,Z \ra=1$ and) such that $II(Z^{\top},Z^{\top})\neq 0.$ 
We recall that $W$ is a lightlike vector field tangent to $M$ such that $\la Z^{\top},W\ra=-1,$ and that $Z^{\perp}$ is a unit vector field normal to $M.$ As before, we consider the unit vector field normal to the surface $$\nu:=\frac{II(Z^{\top},Z^{\top})}{|II(Z^{\top},Z^{\top})|}\ \in\ NM;$$ recall that $\nu$ is orthogonal to $Z^{\perp}$ (see the proof of Lemma \ref{operador de forma}). We moreover suppose that 
\begin{equation}\label{marco movil}
e_1:=\frac{Z^{\top}+W}{\sqrt{2}},\hspace{0.2in} e_2:=\frac{Z^{\top}-W}{\sqrt{2}}, \hspace{0.2in} e_3:= Z^{\perp} \hspace{0.1in}\mbox{and}\hspace{0.1in} e_4:=\nu, 
\end{equation} is an oriented and orthonormal basis of $\R^{3,1},$ and define the orthonormal basis \eqref{base bivectores} of $\Lambda^2\R^{3,1},$ with respect to the form $H.$

\begin{lem}\label{G dG}
The Gauss map of $M$ is given by $G=W\wedge Z^{\top},$ and satisfies 
\begin{equation*}
dG(Z^{\top})=-\vec{H}\wedge Z^{\top}+W\wedge II(Z^{\top},Z^{\top}) \hspace{0.1in}\mbox{and}\hspace{0.1in} dG(W)=\vec{H}\wedge W-Z^{\top} \wedge II(W,W).
\end{equation*}
\end{lem}
\begin{proof} We only need to compute   \begin{equation*}
G=e_1\wedge e_2=\frac{1}{2} (Z^{\top}+W)\wedge (Z^{\top}-W)=W\wedge Z^{\top}.
\end{equation*} The differential of the expression above is given by 
\begin{equation*}
dG(u)=(\nabla_u W+II(W,u))\wedge Z^{\top}+W\wedge (\nabla_u Z^{\top}+II(Z^{\top},u))
\end{equation*} for all $u\in T_pM;$ using the identities of Lemma \ref{simbolos} we conclude the result.
\end{proof}

We define the bivectors 
\begin{equation*} 
N_1:=\frac{1}{\sqrt{2}}(e_2\wedge e_3 + e_3\wedge e_1) \hspace{0.2in}\mbox{and}\hspace{0.2in}
N_2:=\frac{1}{\sqrt{2}}(-e_2\wedge e_3 + e_3\wedge e_1);
\end{equation*}
$N_1$ and $N_2$ are linearly independent, they satisfy $H(N_1,N_1)=H(N_2,N_2)=0$ and $H(N_1,N_2)=-1.$
Explicitly, $N_1$ and $N_2$ are given by 
\begin{equation*}
N_1= Z^{\perp} \wedge W \hspace{0.2in}\mbox{and}\hspace{0.2in} N_2= Z^{\perp}\wedge Z^{\top};
\end{equation*} moreover, with respect to the complex structure $i=-\star$ defined on $\Lambda^2\R^{3,1},$ of a direct computation we get
\begin{equation*}
i N_1= W\wedge \nu \hspace{0.2in}\mbox{and}\hspace{0.2in} i N_2=-Z^{\top}\wedge \nu,
\end{equation*} and the volume element is given by $-iN_1\wedge N_2.$

\begin{lem}\label{dGZ dGW} We have the following identities:
\begin{itemize}
\item $dG(Z^{\top})=i |II(Z^{\top},Z^{\top})| N_1-i\la \vec{H},\nu\ra N_2,$ 
\item $dG(W)= -i\la \vec{H},\nu\ra N_1 + \left( \frac{K_N}{|II(Z^{\top},Z^{\top})|}+i \frac{|\vec{H}|^2-K}{|II(Z^{\top},Z^{\top})|} \right) N_2.$
\end{itemize}
\end{lem}
\begin{proof} Since $0=\la II(Z^{\top},W),Z^{\perp}\ra=-\la\vec{H},Z^{\perp}\ra,$ we have 
\begin{equation}\label{mcv}
\vec{H}=\la\vec{H},Z^{\perp}\ra Z^{\perp}+\la \vec{H},\nu\ra\nu=\la \vec{H},\nu\ra\nu,
\end{equation} replacing this, and the relation given in Proposition \ref{relacion 1}
in the identities of Lemma \ref{G dG}, we easily get the result.
\end{proof}

\paragraph{The pull-back of the form $H$ by the Gauss map.} The pull-back by the Gauss map $G:M \longrightarrow \mathcal{Q}\subset \Lambda^2\R^{3,1}$ of the form $H$ (defined in \eqref{norma H}) permits to define, for all $p\in M,$ the complex quadratic form on the tangent space $T_pM$
\begin{equation*}
G^*H_p:T_pM \longrightarrow \C, \hspace{0.3in} u \longmapsto H(dG_p(u),dG_p(u)).
\end{equation*} This form is analogous to the third fundamental form of the classical theory of surfaces in Euclidean space. We will describe some properties of this quadratic form for a timelike surface with a canonical null direction.

\begin{lem}\label{HdG} We have the following identities 
\begin{itemize}
\item $H(dG(Z^{\top}),dG(Z^{\top}))=-2|II(Z^{\top},Z^{\top})|\la\vec{H},\nu\ra,$
\item $H(dG(W),dG(W))=-2\frac{1}{|II(Z^{\top},Z^{\top})|}\la\vec{H},\nu\ra \left(|\vec{H}|^2-K-iK_N \right),$
\item $H(dG(Z^{\top}),dG(W))=\left( 2|\vec{H}|^2-K \right)-iK_N.$
\end{itemize}
\end{lem}
\begin{proof} The proof of these equalities is obtained by a direct computation using the expressions of $dG(Z^{\top})$ and $dG(W)$ given in Lemma \ref{dGZ dGW}.
\end{proof}

\begin{prop} The discriminant of the complex quadratic form $G^*H$ satisfies 
$$\mbox{disc}\ G^*H:=-\det G^*H=-(K+iK_N)^2,$$ where $K$ and $K_N$ are the Gauss and normal curvatures of the surface $M.$
\end{prop}
\begin{proof} Using the identities of Lemma \ref{HdG}, by a direct computation we get 
\begin{align*}
\det G^*H&= \left( (2|\vec{H}|^2-K)-iK_N \right)^2-4 |\vec{H}|^2\left(|\vec{H}|^2-K-iK_N \right) \\
&= K^2+2iKK_N-K_N^2,
\end{align*} which implies the result.
\end{proof}

%
\begin{prop}\label{minimal0} The complex quadratic form $G^*H$ is zero at every point of $M$ if and only if $M$ is minimal and has flat normal bundle.
\end{prop}
\begin{proof} We recall that $M$ is minimal if and only if $\la\vec{H},\nu\ra=0$ (identity \eqref{mcv}), and that normal curvature zero implies Gauss curvature zero (see Corollary \ref{normal plano}). Using the identities of  Lemma \ref{HdG}, since $II(Z^{\top},Z^{\top})\neq 0,$ we easily get the result. 
\end{proof} 

The interpretation of the condition $G^*H\equiv 0$ is the following: for all $p$ in $M,$ the space $dG_p(T_pM)$ belongs to 
\begin{equation*}
G(p)+\left\lbrace \xi\in\Lambda^2\R^{3,1} \mid H(G(p),\xi)=0=H(\xi,\xi) \right\rbrace \ \subset \ T_{G(p)}\mathcal{Q};
\end{equation*} this set is the union of two complex lines through $G(p)$ in the Grassmannian $\mathcal{Q}$ of the oriented and timelike planes of $\R^{3,1};$ explicitly, these complex lines are given by
\begin{equation*}
G(p)+\C N_1 \hspace{0.3in}\mbox{and}\hspace{0.3in} G(p)+\C N_2. 
\end{equation*}
In particular, the first normal space in $p$ is $1-$dimensional, \textit{i.e.} the osculator space of the surface is degenerate at every point $p$ of $M.$

\paragraph{Asymptotic directions on the surface.} 
For all $p\in M,$ we consider the real quadratic form 
\begin{equation*}
\delta: T_pM  \longrightarrow  \R, \hspace{0.2in} u \longmapsto dG_p(u)\wedge dG_p(u),
\end{equation*} where $\Lambda^4\R^{3,1}$ is identified with $\R$ by means of the volume element $-iN_1\wedge N_2\simeq 1.$ 
A non-zero vector $u\in T_pM$ defines an {\em asymptotic direction} at $p$ if $\delta(u)=0.$ The opposite of the determinant of $\delta,$ with respect to the metric on $M,$ 
\begin{equation*}
\Delta:=-\det{}\delta,
\end{equation*} is a second order invariant of the surface; $\Delta\leq 0$ if and only if there exists asymptotic directions;  $\Delta$ is negative if and only if the surface admits two distinct asymptotic directions at every point. We refer to \cite[Section 4]{bayard_sanchez} (see also \cite{bayard_patty_sanchez}) for a complete description of the asymptotic directions of a timelike surface in $\R^{3,1}.$ 

We will compute the invariant $\Delta$ and describe the asymptotic directions of a timelike surface with a canonical null direction.

\begin{prop}\label{invariante} At every point of $M$ we have: \begin{equation*}
\Delta=-K_N^2
\end{equation*} where $K_N$ is the normal curvature of $M.$ In particular, there exists asymptotic directions at every point of $M.$
\end{prop}
\begin{proof} Since $\delta$ is the imaginary part of the quadratic form $G^*H,$ we have $\delta(Z^{\top})=0$ (first equality of Lemma \ref{HdG}). Using the identities of Lemma \ref{dGZ dGW}, 
by a direct computation we get 
\begin{align*}
\Delta =-\left[ dG(Z^{\top})\wedge dG(W)\right]^2+\delta(Z^{\top})\delta(W)=-K_N^2
\end{align*} since $dG(Z^{\top})\wedge dG(W)=-K_N\ iN_1\wedge N_2 \simeq K_N.$
\end{proof}

\begin{prop}\label{Z-direccion-asintotica} At every point of $M,$  $Z^{\top}$ is an asymptotic direction; moreover, $W$ is an asymptotic direction if and only if $M$ is minimal or has flat normal bundle.
\end{prop}
\begin{proof} Since $\delta$ is the imaginary part of $G^*H,$ using the identities of Lemma \ref{HdG} we have
\begin{equation*}
\delta(Z^{\top})=0 \hspace{0.3in}\mbox{and}\hspace{0.2in} \delta(W)=-2\frac{1}{|II(Z^{\top},Z^{\top})|}\la\vec{H},\nu\ra K_N
\end{equation*} which implies the results.
\end{proof}

According to Proposition \ref{invariante}, if the normal curvature $K_N$ is not zero, there exists two distinct asymptotic directions at every point of the surface. From Proposition \ref{Z-direccion-asintotica}, $Z^{\top}$ is an asymptotic direction; by a direct computation, we describe  the other asymptotic direction.

\begin{prop} If the surface $M$ has not zero normal curvature, there exists two different asymptotic directions given by 
\begin{equation*}
Z^{\top} \hspace{0.3in}\mbox{and}\hspace{0.3in} \frac{\la \vec{H},\nu\ra}{|II(Z^{\top},Z^{\top})|}Z^{\top}+W
\end{equation*} at every point of the surface.
\end{prop}

\paragraph{Acknowledgements.} The first author was supported by the project FORDECyT: Conacyt 265667.
The second author acknowledges support from DGAPA-UNAM-PAPIIT, under Project IN115017.

\end{document}